\documentclass[12pt]{amsart}
\usepackage{amssymb,latexsym,amsmath,amsfonts,amsthm,amscd,graphicx,url,color,hyperref}
\usepackage{stmaryrd}
\usepackage[shortlabels]{enumitem}

\theoremstyle{plain}

\newtheorem{thm}{Theorem}
\newtheorem{prop}[thm]{Proposition}

\newtheorem{lemma}[thm]{Lemma}

\theoremstyle{definition}
\newtheorem{definition}[thm]{Definition}
\newtheorem{remark}[thm]{Remark}

\newtheorem{example}[thm]{Example}

\newtheorem{que}[thm]{Question}

\newcommand{\CS}{\mathcal{S}}

\title[Boundedness of type of almost Gorenstein monomial curves in $\mathbb{A}^5$]{On the boundedness of the type of an almost Gorenstein monomial curve in $\mathbb{A}^5$}
\author{Alessio Moscariello}



\address[Alessio Moscariello]{Dipartimento di Matematica e Informatica, \ Universit\`a di Catania, \  Viale Andrea Doria 6, 
	95125 Catania,Italy}

\email{alemoscariello@hotmail.it}

\begin{document}

\maketitle

\begin{abstract}
We prove that the Cohen-Macaulay type of an almost Gorenstein monomial curve $\mathcal{C} \subseteq \mathbb{A}^5$ is bounded.
\end{abstract}
\section{Introduction}

Almost Gorenstein rings are a class of Cohen-Macaulay rings, introduced by Barucci and Fröberg in 1997 (cf. \cite{BF}), which are close to Gorenstein rings. This class of rings was originally defined in the context of analytically unramified rings of dimension one, however this definition has been later generalized in the context of one-dimensional rings (cf. \cite{GM}), and then to the higher dimension case (cf. \cite{GTT}). 

It is well-known that a local ring is Gorenstein if and only if it is Cohen-Macaulay and its \emph{Cohen-Macaulay type} is equal to one; in a sense, almost Gorenstein rings are meant to extend Gorenstein rings to arbitrary type. The properties of this class of rings have been studied by several authors in the last two decades (cf. \cite{BDS}, \cite{DS}, \cite{E}, \cite{GMT}, \cite{HW}, \cite{NNW}, \cite{SW} for example); in particular, curious patterns started to appear in works dealing with the Cohen-Macaulay type of almost Gorenstein rings. In the original context of one-dimensional analytically unramified rings, including local rings associated to monomial curves, it is well-known that if the embedding dimension is at most $3$ then the Cohen-Macaulay type is either $1$ (and thus the ring is Gorenstein) or $2$ (cf. \cite{FGH}), while if the embedding dimension is at least $4$, there is no upper bound for the Cohen-Macaulay type (cf. \cite[Example p. 75]{FGH}). On the other hand, if we restrict ourselves to almost Gorenstein rings of embedding dimension $4$, Numata (cf. \cite{Nu}) asked if the Cohen-Macaulay type is at most $3$, which turned out to be true (cf. \cite{Mo}). Families of almost Gorenstein rings with arbitrarily large Cohen-Macaulay type $t$ are present in the literature (cf. \cite{GO}, \cite{GMT}); however, for these families of rings, the embedding dimension $e$ is at least $\frac{t}{2}$ (actually, in this context there is no example in the literature of an almost Gorenstein ring satisfying $t > 2e$). Therefore, contrary to the general context, for almost Gorenstein rings and curves the following question (cf. \cite{Nu}, \cite{Mo}) naturally arises.

\begin{que}\label{QuestionType}
	Let $\mathcal{C}$ be an almost Gorenstein monomial curve. 
	Is the Cohen-Macaulay type $t(\mathcal{C})$ bounded by a function of the embedding dimension $e$?
\end{que}

Since the cases $e \le 4$ have already been solved (cf. \cite{FGH}, \cite{Mo}), we investigate the case $e=5$. Interestingly, we obtain the following result.

\begin{thm}\label{main}
	The Cohen-Macaulay type of an almost Gorenstein monomial curve $\mathcal{C} \subseteq \mathbb{A}^5$ is bounded.
\end{thm}

We provide an explicit bound, that is, $t(\mathcal{C}) \le 473$. However, this bound does not seem optimal; computational evidence, included in the last part of this work, suggests that the bound can be largely improved. Finally, we give an example of almost Gorenstein ring with embedding dimension $6$ and type $14$, satisfying $t > 2e$; this suggests that if $e =6$ it might be harder to bound the Cohen-Macaulay type $t$. Our approach relies on the correspondence between numerical semigroups and monomial curves (\cite{BDF}); we study the class of numerical semigroups associated to almost Gorenstein monomial curves, which are called \emph{almost symmetric}.

\section{Numerical semigroups}

A \emph{numerical semigroup} is a submonoid $\CS$ of $(\mathbb{N},+)$ such that the set $\mathbb{N}\setminus \CS$ is finite. Every numerical semigroup is finitely generated and admits a unique minimal system of generators $\{g_1,\ldots,g_e\}$, such that $\gcd(g_1,\ldots,g_e)=1$. The integer $e=e(\CS)$ is called the \emph{embedding dimension} of $\CS$. The largest natural number not belonging to $\CS$ is called the \emph{Frobenius number} of $\CS$, denoted by $F(\CS)=\max \mathbb{N}\setminus \CS$. Elements of $\CS$ which are smaller than the Frobenius number are called \emph{small elements}, and the set of small elements of a numerical semigroup is denoted by $N(\CS)$, while its cardinality is usually denoted by $n(\CS)$. On the other hand, natural numbers not belonging to $\CS$ are called \emph{gaps} of the numerical semigroup $\CS$, while the set of gaps is denoted by $G(\CS)=\mathbb{N}\setminus \CS$.

We say that an integer $f$ is a \emph{pseudo-Frobenius number} of a numerical semigroup $\CS=\langle g_1,\ldots,g_e \rangle$ if $f \not \in \CS$ and $f+g_i \in \CS$ for every $i=1,\ldots,e$. We denote the set of pseudo-Frobenius numbers by $PF(\CS)$, and define the \emph{type} of $\CS$ as the cardinality of $PF(\CS)$, denoted by $t(\CS)$. Clearly $F(\CS) = \max PF(\CS)$. Given a numerical semigroup $S$, we can define a partial order relation $\le_\CS$ on $\mathbb{Z}$ in this way: $x \le_\CS y$ if $y-x \in \CS$. By definition, $PF(\CS)$ is the set of maximal elements of $\mathbb{Z}\setminus \CS$ with respect to $\le_\CS$.

\begin{remark}\label{largetype}
	\begin{enumerate}
		\item In \cite{FGH} it is proved that the type of a numerical semigroup with embedding dimension $3$ is at most $2$.
		\item In \cite[Example p.75]{FGH} it is shown that there exist families of numerical semigroups with embedding dimension $4$ and arbitrarily large type. Namely, if $n \ge 2$ and $r \ge 3n+2$ are given, $s=r(3n+2)+3$ and $\mathcal{S}=\langle s, s+3, s+3n+1, s+3n+2 \rangle$, then $t(\mathcal{S})= 3n+2$.
	\end{enumerate}
\end{remark}

A numerical semigroup $\CS$ is said to be \emph{symmetric} if for every $x \not \in \CS$, $F(\CS)-x \in \CS$; it is well-known that symmetric numerical semigroups are exactly those with type equal to $1$. Further, we say that $\CS$ is \emph{almost symmetric} if for every $x \not \in \CS$ we have either $F(\CS)-x \in \CS$ or $\{x, F(\CS)-x\} \subseteq PF(\CS)$. Almost symmetric numerical semigroups can have arbitrarily large type (cf. \cite{GO}); however, contrary to the general case (see Remark \ref{largetype}), it is known that almost symmetric numerical semigroups with embedding dimension $4$ have type at most $3$ (cf. \cite{Mo}). It is not clear if, once the embedding dimension $e$ of an almost symmetric numerical semigroup is fixed, there is some sort of restriction which bounds $t$ in function of $e$ (see Question \ref{QuestionType}). 

The monograph \cite{RG1} is a good reference on numerical semigroups. As one might guess from the terminology, there is a strict relation between numerical semigroups and commutative rings. In fact, the value semigroups associated to analytically unramified one dimensional local domains are actually numerical semigroups, and there is a correspondence between the invariants of these mathematical objects (cf. \cite{BDF}). For our purposes, it is worth remembering that the Cohen-Macaulay type $t(\mathcal{C})$ of an almost Gorenstein monomial curve coincides with the type $t(\mathcal{\CS})$ of the associated numerical semigroup, which is almost symmetric.

\section{Type of almost symmetric numerical semigroups}
Let $\CS= \langle g_1,\ldots,g_e \rangle$ be a numerical semigroup.

Let $f \in \mathbb{Z}$. We say that a square matrix $A=(a_{ij})$ of order $e$ is a RF-matrix (short for \emph{row-factorization matrix}) for $f$ if, for every $i=1,\ldots,e$, $f-a_{ii}g_i \in S$, $f- (a_{ii}+1)g_i \not \in \CS$, $a_{ij} \ge 0$ for every $j \neq i$ and $\displaystyle f=\sum_{j=1}^{e} a_{ij}g_j$.

\begin{remark}
	It is straightforward to check that $f \in \CS$ if and only if there is a RF-matrix for $f$ such that $a_{ii} \ge 0$ for every $i=1,\ldots,e$, while $f \in PF(\CS)$ if and only if there is a RF-matrix for $f$ such that $a_{ii} = -1$ for every $i=1,\ldots,e$.
\end{remark}

\begin{example}
Consider the numerical semigroup $\CS= \langle 5, 12, 13 \rangle$, and let $f= 19$. Since we can write $19 = - 5 + 2 \cdot 12 = 5 - 12 + 2 \cdot 13 = 4 \cdot 5 + 12 - 13$, we have that $$A = \begin{pmatrix}
	-1 & 2 & 0\\
	1 & -1 & 2 \\
	4 & 1 & -1
\end{pmatrix} $$

is a RF-matrix for $f$, and furthermore we can see that $f \in PF(\CS)$. 
\end{example}

\begin{prop}\label{prop1}\protect{\cite[Proposition 4]{Mo}}
	Let $f, F(\CS)-f \in PF(\CS)$, $A=(a_{ij})$ be a RF-matrix for  $f$ and $B=(b_{ij})$ a RF-matrix for $F(\CS)-f$. Then for every $i \neq j$ we have $a_{ij}b_{ji}=0$.
	Therefore, there are at least $e(e-1)$ entries of $A \cup B$ equal to zero.
\end{prop}

For every $i,j \in \{1,\ldots,e\}$, $i \neq j$, define $\lambda_{ij}=\max \{ k \in \mathbb{N} \ | \ kg_j-g_i \not \in S\} \ge 2$ and $\Lambda_{ij} = \lambda_{ij} g_j-g_i$, and consider the multiset $\Lambda := \{ \Lambda_{ij} \ | i \neq j\}$. Notice that $|\Lambda| = e(e-1)$.

\begin{remark}\label{lambda}
	For every $i,j \in \{1,\ldots,e\}$, $i \neq j$, if $kg_j-g_i$ is contained in $PF(\CS)$,
	then $k$ must be $\lambda_{ij}$, which implies that such an element is unique and belongs to $\Lambda$. 
\end{remark}

\begin{remark}
	As a direct consequence of Proposition \ref{prop1} we can see that, if $\CS$ is an almost-symmetric numerical semigroup with embedding dimension $4$ and $A,B$ are RF-matrices for two pseudo-Frobenius numbers $f,F(\CS)-f$, there	are at least $4\cdot 3 = 12$ entries of $A \cup B$ equal to zero, therefore there is at least one row of $A$ or $B$ with exactly one positive entry. This implies that either $f$ or $F(\CS)-f$ (or both) has to be equal to an element of $\Lambda$, meaning that there are a bounded number of couples $\{f,F(\CS)-f\}$ of elements of $PF(\CS)$. This simple argument shows that the type of an almost symmetric numerical semigroup with embedding dimension $4$ is bounded (cf. \cite{Mo}).
\end{remark}

Let $\CS=\langle g_1,\ldots,g_5 \rangle$ be an almost-symmetric numerical semigroup with embedding dimension five.

Since $\CS$ is almost symmetric, all its pseudo-Frobenius numbers $f$ besides $F(\CS)$ are such that $F(\CS)-f \in PF(\CS)$. We say that a pseudo-Frobenius number $f \in PF(\CS) \setminus \{F(\CS)\}$ is \emph{good}  if either $f$ or $F(\CS)-f$ is of the form $kg_j-g_i$ for some $i \neq j$ and some positive integer $k$, otherwise we say that $f$ is \emph{bad}. We denote by $PF_g(\CS)$ the set of good pseudo-Frobenius numbers, and by $PF_b(\CS)$ the set of bad pseudo-Frobenius numbers.
\begin{remark}\label{typerem}
	By definition $f \in PF_g(\CS)$ if and only if $F(\CS)-f \in PF_g(\CS)$, and $t(\CS)=|PF_g(\CS)|+|PF_b(\CS)|+1$. 
	
\end{remark}

\begin{prop}\label{good}
	Let $\CS$ be an almost-symmetric numerical semigroup with embedding dimension five. Then $|PF_g(\CS)| \le 40$.
\end{prop}

\begin{proof}
	In light of Remark \ref{typerem} we can partition the set $PF_g(\CS)$ in couples $\{f, F(\CS)-f\}$ (eventually pairing the element $\frac{F(\CS)}{2}$ with itself). On the other hand, by Remark \ref{lambda} every couple $\{f,F(\CS)-f\}$ is associated to (at least) one element of $\Lambda$, thus $|PF_g(\CS)| \le 2 |\Lambda| = 40$.
\end{proof}

Let $f \in PF_b(\CS)$,  $A=(a_{ij})$ be a RF-matrix for $f$, and $B=(b_{ij})$ a RF-matrix for $F(\CS)-f$. Notice that if at least one row of either $A$ or $B$ has at least $3$ entries equal to zero, then by definition either $f$ or $F(\CS)-f$ would be of the form $K g_j-g_i$ for some $i,j$, contradicting $f \in PF_b(\CS)$. But since by Proposition \ref{prop1} there are at least $20$ zeroes among the entries of $A$ and $B$, we deduce that $A$ and $B$ must satisfy the following properties:
\begin{enumerate}[(a)]
	\item  There are exactly $20$ entries of $A \cup B$ equal to zero, and for every $i \neq j$ exactly one of $a_{ij}$ and $b_{ji}$ is zero.
	\item Each row and column of $A$ and $B$ has exactly two positive entries and exactly two zeroes.
\end{enumerate}

\begin{example}
	Let $\CS= \langle 64, 67, 91, 138, 150 \rangle$, $PF(\CS)=\{209,327,445,654\}$. It is simple to check that $209$ is a good pseudo-Frobenius number, since $209=3 \cdot 91 - 64$. On the other hand, $f=327$ is a bad pseudo-Frobenius number, and its RF-matrix is $$A=\begin{pmatrix}
		-1 & 0 & 1 & 0 & 2\\
		4 & -1 & 0 & 1 & 0 \\
		0 & 4 & -1 & 0 & 2 \\
		3 & 0 & 3 & -1 & 0 \\
		0 & 3 & 0 & 2 & -1
	\end{pmatrix}.$$
Since $F(\CS)=654$, we have $f=F(\CS)-f$ and $B=A$ in the previous argument; it is easy to see that the couple of RF-matrices $A$ and $B=A$ satisfy both properties (a) and (b).
\end{example}

\begin{definition}
	Let $M=(m_{ij}),M'=(m'_{ij})$ be two square matrices of order $n$. We say that $M$ and $M'$ have the same \textbf{$0$-configuration} if $m_{ij}=0$ if and only if $m'_{ij}=0$. 
\end{definition}

In our context, if $A=(a_{ij})$ and $A'=(a'_{ij})$ are two RF-matrices for the same bad pseudo-Frobenius number $f \in PF_b(\CS)$ and $B=(b_{ij})$ is a RF-matrix for $F(\CS)-f$, then by (a) in both couples $\{a_{ij},b_{ji}\}$ and $\{a'_{ij},b_{ji}\}$ there is exactly one zero: therefore we have that $a_{ij}=0$ if and only if $a'_{ij}=0$. Then, all RF-matrices of a bad pseudo-Frobenius $f$ have the same $0$-configuration. 

\begin{lemma}\label{rows}
	Let $f \in PF_b(\CS)$ and $A=(a_{ij})$ be a RF-matrix for $f$. Then there are two row vectors $R_1$ and $R_2$ of $A$ such that, for exactly one index $j$, the $j$-th entries of $R_1$ and $R_2$ are both positive.
\end{lemma}

\begin{proof}
	The RF-matrix $A$ satisfies the two properties (a) and (b), hence each column of $A$ contains exactly two positive entries and two zeroes. Then, for every $j=1,\ldots,5$, there exists a couple of rows $\{R_{1j},R_{2j}\}$ which $j$-th entries are both positive. However, if the other positive entry of $R_{1j},R_{2j}$ corresponds to the same index $k$, then the two indices $j$ and $k$ would be associated to the same couple $\{R_{1j},R_{2j}\}$, and they are the only two indices associated to these two rows (each row has exactly two positive entries). However, since the number of indices is odd, this cannot happen for all values: therefore there must be one index $j$ such that the two rows $R_{1j},R_{2j}$ have exactly one positive common component, proving our thesis.  
\end{proof}

\begin{lemma}\label{key}
	
	Let $f,f' \in PF_b(\CS)$ be such that 
	\begin{alignat}{2}
		f &= a_{ij}g_j+a_{ik}g_k-g_i &= a_{pj}g_j+a_{pq}g_q-g_p \\ 
		f' &= b_{ij}g_j+b_{ik}g_k-g_i &= b_{pj}g_j+b_{pq}g_q-g_p 
	\end{alignat}
	where the indices $i,j,k,p,q \in \{1,\ldots,5\}$ are such that $i \neq j$, $i \neq k$, $j \neq k$, $p \neq j$, $p \neq q$, $j \neq q$ (i.e. all indices on the same side of an equation are distinct), $q \neq k$, and all coefficients $a_{mn}$ and $b_{mn}$ are positive integers, with $a_{ij} \ge a_{pj}$, $b_{ij} \ge b_{pj}$. 
	
	Then $f=f'$.
\end{lemma}

\begin{proof}
	Assume without loss of generality that $a_{pq} \ge b_{pq}$. From equation (2) we obtain $$b_{pq}g_q=(b_{ij}-b_{pj})g_j+b_{ik}g_k+g_p-g_i.$$
	
	By substitution in $f$ we get
	\begin{alignat}{2} 
		f &=a_{pj}g_j+a_{pq}g_q-g_p \nonumber \\ &=a_{pj}g_j+(a_{pq}-b_{pq})g_q-g_p+(b_{ij}-b_{pj})g_j+b_{ik}g_k+g_p-g_i \nonumber \\ &=(a_{pj}+b_{ij}-b_{pj})g_j+b_{ik}g_k+(a_{pq}-b_{pq})g_q-g_i.
	\end{alignat}
	Since $j \neq k$, $j \neq q$ and $q \neq k$, and $a_{pj}+b_{ij}-b_{pj} \ge a_{pj} > 0$, $b_{ik} > 0$ and $a_{pq}-b_{pq} \ge 0$, equation (3) is associated to a row of a RF-matrix for $f$. Since $f \in PF_b(\CS)$, by property (b) this row must contain exactly $2$ positive entries and two zeroes, thus we must necessarily have $a_{pq}=b_{pq}$ (notice that if $q=i$ we still have $a_{pq}=b_{pq}$ because $f \not \in \CS$).
	
	Therefore $f-f'=(a_{pj}-b_{pj})g_j$, and thus we have either $f \le_S f'$ or $f' \le_S f$. In both cases, since $f, f' \in PF(\CS)$ are maximals, we conclude that $f=f'$.
\end{proof}

\begin{prop}\label{bad}
	Let $\CS$ be an almost-symmetric numerical semigroup with embedding dimension five, and let $N$ be the number of possible $0$-configurations satisfying the two properties (a) and (b). Then $|PF_b(\CS)| \le 2N$.
\end{prop}

\begin{proof}
	We argue by contradiction: assume that $|PF_b(\CS)| \ge 2N + 1$. By definition of $N$, there must be three distinct pseudo-Frobenius numbers $f_1,f_2,f_3 \in PF_b(\CS)$ which RF-matrices have the same $0$-configuration. Let $A$ be a RF-matrix for $f_1$; by Lemma \ref{rows} there exist two rows $R_1,R_2$ of $A$ such that for exactly one index $j$ the $j$-th entries of $R_1$ and $R_2$ are both positive. Since the RF-matrices of $f_2$ and $f_3$ have the same $0$-configuration, we obtain
	\begin{alignat}{2}
		f_1 &= a_{ij}g_j+a_{ik}g_k-g_i &= a_{pj}g_j+a_{pq}g_q-g_p \nonumber\\
		f_2 &= b_{ij}g_j+b_{ik}g_k-g_i &= b_{pj}g_j+b_{pq}g_q-g_p \nonumber\\
		f_3 &= c_{ij}g_j+c_{ik}g_k-g_i &= c_{pj}g_j+c_{pq}g_q-g_p \nonumber
	\end{alignat}
	and since the expressions correspond to positive entries of $R_1$ and $R_2$ we deduce that all indices appearing on the same side of these equations are distinct, while the hypothesis that $R_1$ and $R_2$ have exactly one positive entry in the same index $j$ means that $q \neq k$. 
	
	Consider now the three couples of elements $\{a_{ij},a_{pj}\}$, $\{b_{ij},b_{pj}\}$, $\{c_{ij},c_{pj}\}$. Obviously, at least two of these couples must be ordered in the same way, therefore (up to a change of indices) we obtain that $a_{ij} \ge a_{pj}$ and $b_{ij} \ge b_{pj}$. Hence $f_1,f_2$ fall under the hypotheses of Lemma \ref{key}, implying $f_1=f_2$, yielding a contradiction.
\end{proof} 

Combining Propositions \ref{good} and \ref{bad} we obtain $t(\CS) \le 2N+41$, and since $N$ is bounded (for instance it is easy to check that $N < 6^5$), we prove the following Theorem, which is equivalent to Theorem \ref{main}. 
\begin{thm}\label{type}
	The type of an almost symmetric numerical semigroup $\CS$ with embedding dimension five is bounded.
\end{thm}

\begin{remark}\label{bound}
	We computed with GAP all $0$-configurations of square matrices $A=(a_{ij})$ of order $5$ satisfying property (b) - i.e. each row and column of $A$ contains exactly two zeroes and two positive entries, and $a_{ii}=-1$. We obtained that there are exactly $216$ such configurations, thus $N \le 216$. In particular, this yields $t(\CS) \le 2 \cdot 216 + 41 = 473$.
\end{remark}

The next remark suggests that the bound obtained in Remark \ref{bound} can be largely improved, though it would require new techniques. 

\begin{remark}\label{comp}
	An extensive computation performed by P.A. Garc\'ia-S\'anchez, using GAP and the package Numericalsgps (see \cite{DGM}), found that, if $\CS=\langle g_1,\ldots,g_5 \rangle$ is an almost symmetric numerical semigroup and $g_1<\ldots<g_5 \le 20 0$ (this family consists of more than $7 \cdot 10^6$ examples!), then $t(\CS) \le 5$. Moreover, there is always at most one bad pseudo-Frobenius number, which is exactly $\frac{F(\CS)}{2}$. If this finding turns out to be true for all almost symmetric numerical semigroups $\CS$ with embedding dimension $5$, the bound could be vastly improved.
\end{remark}

The first key point of our argument is that RF-matrices associated to \emph{bad} pseudo-Frobenius numbers must have some very special properties. In the higher dimension case, these matrices are hard to control. To see this, let $\CS$ be an almost symmetric numerical semigroup with embedding dimension $6$, $f, F(\CS)-f$ be bad pseudo-Frobenius numbers and $A,B$ be RF-matrices for $f,F(\CS)-f$ respectively. Then Proposition \ref{prop1} infers that there are at least $30$ entries equal to zero on $A \cup B$; however, contrary to the embedding dimension $5$ case, it could be possible to have more than $30$ without having too many zeroes on each row (a pseudo-Frobenius number is good if at least one row contains $e-2$ zeroes). Moreover, the key idea behind Proposition \ref{bad} is that, under our hypotheses, there cannot be more than two bad pseudo-Frobenius numbers sharing the same $0$-configuration. This might not be true in the higher dimension case.

\begin{example}
	Consider the numerical semigroup $\CS=\langle 455,497,574,589,631,708 \rangle$. We have $$PF(\CS)=\{ 3079, 3289, 3521, 3655, 3674, 3789, 3923, 4057, 4172, 4191, 4325, 4557, 4767, 7846 
	\},$$ thus $\CS$ has embedding dimension $e=6$ and type $t(\CS) = 14 > 2e$. Further, $PF(\CS)$ contains the arithmetic progression $\{3521,3655,3789,3923,4057,4191,4325\}$ of ratio $134$. Notice that, if $A_f$ is a RF-matrix associated to $f \in PF(\CS)$, we have
	$$A_{3521} = \begin{pmatrix}
		-1 & 8 & 0 & 0 & 0 & 0 \\
		0 & -1 & 7 & 0 & 0 & 0  \\
		9 & 0 & -1 & 0 & 0 & 0 \\
		0 & 7 & 0 & -1 & 1 & 0 \\
		0 & 0 & 6 & 0 & -1 & 1 \\
		8 & 0 & 0 & 1 & 0 & -1
	\end{pmatrix},$$ $$A_{3521+d\lambda} = \begin{pmatrix}
		-1 & 8-\lambda & 0 & 0 & \lambda & 0 \\
		0 & -1 & 7-\lambda & 0 & 0 & \lambda  \\
		9-\lambda & 0 & -1 & \lambda & 0 & 0 \\
		0 & 7-\lambda & 0 & -1 & \lambda+1 & 0 \\
		0 & 0 & 6-\lambda & 0 & -1 & \lambda+1 \\
		8-\lambda & 0 & 0 & \lambda+1 & 0 & -1
	\end{pmatrix}
	$$
	
	for $\lambda=0,1,\ldots,6$.
	
	The $5$ pseudo-Frobenius numbers $\{3521,3655,3789,3923,4057,4191\}$ (obtained for $\lambda=1,\ldots,5$) all have RF-matrices sharing the same $0$-configuration.
	
\end{example}

\section*{Acknowledgements}
I would like to thank Marco D'Anna, Alessio Sammartano and Francesco Strazzanti for many helpful feedbacks and discussions on this work. Moreover, I would like to thank P.A. Garc\'ia-S\'anchez for his help with the computations in Remark \ref{comp}. Finally, I would like to thank the referees for their suggestions.

The author was funded by the project “Propriet\`a locali e globali
di anelli e di variet\`a algebriche”-PIACERI 2020–22, Universit\`a degli Studi
di Catania.

\end{document}